\documentclass{amsart}

\usepackage{amsmath,amssymb,amsthm}
\usepackage[dvipdfm]{graphicx}
\usepackage{enumerate}
\usepackage[dvips]{color}
\usepackage{version}

\newtheorem{Thm}{Theorem}[section]

\newtheorem{Cor}[Thm]{Corollary}
\newtheorem{Prop}[Thm]{Proposition}

\newtheorem{``Conj"}[Thm]{``Conjecture"}

\newtheorem{Claim}[Thm]{Claim}
\newtheorem{Fact}[Thm]{Fact}

\theoremstyle{remark}
\newtheorem{Rem}[Thm]{Remark}

\theoremstyle{definition}
\newtheorem{Def}[Thm]{Definition}

\newtheorem*{ack}{Acknowledgements}

\newcommand{\Supp}{\mathop{\mathrm{Supp}}\nolimits}

\newcommand{\DF}{\mathop{\mathrm{DF}}\nolimits}

\begin{document}

\title[Log K-stability]
{Testing Log K-stability by blowing up formalism}

\author{Yuji Odaka}
\address{Research Institute for Mathematical Sciences, Kyoto University,
Kyoto, Japan}
\email{yodaka@kurims.kyoto-u.ac.jp}

\author{Song Sun}
\address{Department of Mathematics, Imperial College,  London, United Kingdom}
\email{s.sun@imperial.ac.uk}

\date{6th December, 2011}

\maketitle

\begin{abstract}
We study logarithmic K-stability for pairs by extending the formula for Donaldson-Futaki invariants to log setting. 
We also provide algebro-geometric counterparts of recent results of existence of K\"ahler-Einstein metrics with cone singularities. 
\end{abstract}
\tableofcontents

\section{Introduction}

One of the central issue of recent developments of K\"ahler geometry is 
on the conjectural relationship between the existence of ``canonical" K\"ahler metrics and 
stability in certain sense. Along that line, K-stability is first defined by Tian \cite{Tia97} 
and later generalized by Donaldson \cite{Don02}. 

The logarithmic K-stability with parameter $\beta\in(0,1]$ is defined in \cite{Don11} which conjecturally corresponds to the existence of K\"ahler-Einstein metrics with cone angle $2\pi\beta$ along the  divisor on Fano manifolds. 

The purpose of this paper is to extend most of results for K-stability given in 
\cite{Od08}, \cite{Od11a}, \cite{Od11b}, \cite{Od11c}, \cite{OS10} to this 
logarithmic setting which concerns a pair $(X, D)$. 
On the way, we  extend
\cite[Theorem 1.1]{Sun11} purely algebraically, allowing more general anti-canonical divisors. We also recover algebraic counterparts of  \cite[Theorem 1.8]{Ber10}, \cite[Theorem 1.1]{Bre11} and \cite[Theorem 2]{JMR11}. This provides more  evidence for the above logarithmic Yau-Tian-Donaldson conjecture.

We expect that these will have meanings even in the absolute case. 
On the one hand, Donaldson has  recently proposed an approach of constructing K\"ahler-Einstein metrics on Fano manifolds by deforming  K\"ahler-Einstein metrics with 
edge singularities along the anti-canonical divisor. On the other hand, the minimal model program (MMP, for short)  is nowadays naturally studied in  log setting as it is also useful to absolute case study, giving an inductive framework on dimension based on adjunction argument. We expect and partially prove in this paper 
that the relation with stability and birational geometric framework based on MMP 
(which first appeared in \cite{Od08} and developed in \cite{Od11b}, \cite{LX11} etc) 
fits this expectation. 

We make several remarks here. First, about the pathology in \cite{LX11}; they 
pointed out the necessity to restrict attention to 
only test configurations which {\it satisfy the $S_2$ condition 
(or normality, for normal original variety)} for the definition of K-stability. It does not violate our arguments and actually it is compatible with the framework of \cite{Od11a},  as we explain later in section \ref{bl.up.sec}. Second, although we argue about a variety $X$ with an \textit{integral} divisor $D$ unless otherwise stated, the following argument mostly works to give extension to the case where $D$ can be a $\mathbb{Q}$-divisor. 

We work over $\mathbb{C}$, the complex number field, although a large part 
of arguments in this paper works over more general fields as it is purely 
algebro-geometric. We will use linear equivalence class of Cartier divisor, 
invertible sheaf, line bundle interchangeably. 

\begin{ack}
Both authors would like to  thank Professor Simon Donaldson for his helpful advice. 
This joint work started when both authors attended 
``2011 Complex geometry and symplectic geometry conference" held at the University of Science and Technology of China, Hefei, in August, 2011. Later, the  discussion developed when the first author visited the Imperial College London in November, 2011. 
Both authors are grateful to Professor Xiu-xiong Chen 
for invitation to the conference and nice hospitality. 
The first author is also grateful to Professor Simon Donaldson for making the opportunity of visiting the Imperial College London, and to both Professors Simon Donaldson and Richard Thomas for warm hospitality. 

During the preparation of this paper, we are informed by C. Li and C. Xu that in the article \cite{LX} to appear, they also prove independently Theorem \ref{CY.gen} and a version of Theorem \ref{DF.formula}, and their approach is quite different.

Y.O is partially supported by the Grant-in-Aid for Scientific Research (KAKENHI No.\ 21-3748) and the Grant-in-Aid for JSPS fellows. 
\end{ack}

\section{Preliminary}

\subsection{Basics of discrepancy}\label{disc}

Consult \cite{KM98} for the details. 
Along the development of the (log) minimal model program in a few decades, 
some mild singularities should have been admitted on varieties in concern. 
On that way, the theory of discrepancy and some mild singularities' class developed. 
We review these as the efficiency of that theory in the  study of stability turned out in \cite{Od08} and developed in \cite{Od11a}, \cite{Od11b}, \cite{LX11} etc which we follow. 

Assume $(X,D)$ to be a pair of a normal variety $X$ and an effective $\mathbb{Q}$-divisor $D$ in this section. $D$ is usually referred as a boundary divisor. Let $\pi\colon X'\to X$ be a log resolution of $D$, i.e., $\pi$ is a proper birational morphism such that $X'$ is smooth and the divisor $\pi^*D+E$ has a simple normal crossing support, where $E$ is the exceptional divisor of $\pi$.
Let $K_{X'/X}:=K_{X'}-\pi^*K_{X}$. 
Then, we denote 
\begin{equation*}\label{eq:discrepancy}
	K_{X'}-\pi^*(K_X+D)=\sum a_{i}E_{i}, 
\end{equation*} 
where $a_{i} \in \mathbb{Q}$ and $E_{i}$ runs over the set of divisors of $X'$ supported on the exceptional locus or the support $\Supp(\pi^{-1}_{*}D)$ of the strict transform of $D$. $a_i$'s are the so called {\it discrepancy} (of $(X,D)$ for $E_i$) which measures 
the mildness of singularities. We usually write it as $a(E_i;(X,D)):=a_i$. 
We also note that under the above situation, we can also consider discrepancy of pair $(X,D+cI)$ attached with additional coherent ideal $I\subset \mathcal{O}_{X}$ 
multiplied formally by some real number $c$ as follows: 
$$
a(E_i;(X,D+cI)):=a(E_i;(X,D))-c{\rm val}_{E_i}(I)
$$ 
\noindent 
where ${\rm val}_{E_i}$ means valuation of ideal $I$ measured by $E_i$. 
(In that case, we might get discrepancies which are only {\it real} numbers. )
Under above notation, we define some classes of mild singularities as follows: 

\begin{Def}
The pair $(X,D)$ is called \textit{log canonical} (lc, for short) 
if and only if $a_{i}\ge -1$ for any $E_{i}$. 
\end{Def}
\noindent 
From the definition, for $(X,D)$ to be log canonical, all the coefficients of $D$ should be at most $1$. 

For a stronger notion, log terminality, we have following versions for pairs. 
\begin{Def}
{\rm (i)}
The pair $(X,D)$ is called \textit{kawamata log terminal} (klt, for short)
if and only if $a_{i}> -1$ for any $E_{i}$. 
{\rm (ii)}
The pair $(X,D)$ is called \textit{purely log terminal} (plt, for short)
if and only if $a_{i}> -1$ for any exceptional $E_{i}$. 
\end{Def}

\noindent
If we allow negative coefficients for $D$, these conditions give definitions of 
{\it sub} kawamata-log-terminality (resp.\ {\it sub} log canonicity). 
Note these definitions are independent of the choice of log resolution. 

One advantage of considering pairs is the following inversion of adjunction, 
which relates mildness of singularity of pair to that of their boundary. 

\begin{Thm}[{Inversion of adjunction \cite[section 5.2]{KM98}, \cite{Kaw07}}]\label{IA}
Assume $D$ is decomposed as $D=D'+D''$ where $D'$ is an effective integral reduced normal Cartier divisor and $D''$ is also an effective $\mathbb{Q}$-divisor which has 
no common components with $D'$. Then the followings hold. 

{\rm (i)}$(X,D)$ is purely log terminal on some open neighborhood of $D'$ 
if and only if $(D',D''|_{D'})$ is kawamata log terminal. 

{\rm (ii)}$(X,D)$ is log canonical on some open neighborhood of $D'$ if and only if $(D',D''|_{D'})$ is log canonical. 
\end{Thm}

\noindent
Note that there are generalizations to the case where 
$D'$ is not necessarily  Cartier nor normal, for which we need to think over normalization of $D'$ 
with extra divisor involved. Consult \cite[section 5.2]{KM98}, \cite{Kaw07}, \cite[Corollary 1.2]{OX11} for the details and proofs. 

\subsection{Seshadri constant}

Let $X$ be a quasi-projective variety and 
$J\subset \mathcal{O}_X$ be a coherent ideal on $X$. 
The {\it Seshadri constant} 
of $J$ with respect to an ample line bundle $L$ 
is defined by 
$$
	\mathrm{Sesh}(J;(X, L)):=\sup \{c>0\mid \pi^*L(-cE)\,\,\mbox{is ample}\},
$$
where $\pi\colon X'\to X$ is the blow up of $X$ along $J$. 
It is named after Seshadri's ampleness criterion and defined by Demailly in early 
90s after it for his approach to Fujita's conjecture about positivity of adjoint line  bundles. For the study of stability of varieties it first appeared in \cite{RT07} to define slope stability and its extension played a crucial role in \cite{OS10} etc to judge K-stability. We extend this picture further in section \ref{sec.alpha}. 

\section{A framework to work on log Donaldson-Futaki invariants}\label{bl.up.sec}

In this section, after reviewing the definition of log Donaldson-Futaki invariants \cite{Don11}, we extend the framework of \cite{Od11a} to the logarithmic setting. 
The definition of logarithmic K-stability is as follows. 

\begin{Def}[{\cite{Don11}}]
Suppose that $(X,L)$ is a $n$-dimensional polarized variety, 
and $L$ is an ample line bundle. 
Also suppose $D$ is an effective integral reduced divisor on $X$. Then, 
a {\it log test configuration} (resp.\ {\it log semi-test configuration}) 
of $((X,D),L)$ consists of  a pair of test configurations (resp.\ {\it semi-test configurations}) 
$(\mathcal{X},\mathcal{L})$ for $(X,L)$ 
and $(\mathcal{D},\mathcal{L}|_{\mathcal{D}})$ 
for $(D,L|_D)$ with the same exponent inside 
$(\mathcal{X},\mathcal{L})$ and a compatible $\mathbb{G}_m$-action. 
\end{Def}

\noindent
From the definition,  we have $\mathcal{D}=\overline{\mathbb{G}_m(D\times \mathbb{A}^1)}\subset \mathcal{X}$. For the definition of (log) Donaldson-Futaki 
invariants, we prepare the following notation: 

\begin{itemize}
\item $\chi(X,L^{\otimes m})=a_0m^{n}+a_1m^{n-1}+O(n-2)$, 

\item $w(m):=$ total weight of $\mathbb{G}_m$ action on 
$H^{0}(\mathcal{X}|_{\{0\}},\mathcal{L}|_{\{0\}}^
{\otimes m})\\ =b_0m^{n+1}+b_1m^{n}+O(n-1)$, 

\item$\chi(D,L|_D^{\otimes m})=\tilde{a}_0m^{n-1}+\tilde{a}_1m^{n-2}+O(n-3)$, 

\item $\tilde{w}(m):=$ total weight of $\mathbb{G}_m$ action on 
$H^{0}(\mathcal{D}|_{\{0\}},\mathcal{L}|_{\{\mathcal{D}_0\}}^
{\otimes m})\\ =\tilde{b}_0m^{n}+\tilde{b_1}m^{n-1}+O(n-2)$. 
\end{itemize}

\noindent
Here, $O(-)$ stands for the Landau symbol. 
Recall that the (usual) Donaldson-Futaki invariant of $(\mathcal{X},\mathcal{L})$, 
${\it DF}(\mathcal{X},\mathcal{L})$ is defined as $2(b_0a_1-b_1a_0)$. Here we deliberately add a constant $2$ so that the formula in the following definitions is simplified.

\begin{Def}[{\cite{Don11}}]\label{DF}
The {\it logarithmic Donaldson-Futaki invariant} of  a log test configuration $((\mathcal{X},\mathcal{D}),\mathcal{L})$ with cone angle $2\pi\beta$ ($0\leq\beta\leq 1$) is 
$${\it DF}_{\beta}((\mathcal{X},\mathcal{D}),\mathcal{L})={\it DF}(\mathcal{X},\mathcal{L})+
(1-\beta)(a_0\tilde{b}_0-b_0\tilde{a}_0).$$ 
\end{Def}

\noindent
Let us call $(1-\beta)(a_0\tilde{b}_0-b_0\tilde{a}_0)$ the {\it boundary part}, 
which does not appear for absolute case (i.e. if $\beta=1$).

\begin{Def}[{\cite{Don11}}]\label{Kst}
Assume $0\leq \beta\leq 1$ and $X$ satisfies Serre's $S_2$ condition(Note this is weaker than $X$ being normal). 
$((X,D),L)$ is said to be {\it logarithmically K-stable (resp.\ logarithmically K-semistable) with cone angle $2\pi\beta$} if and only if ${\it DF}_{\beta}((\mathcal{X},\mathcal{D}),\mathcal{L})$ is positive (resp.\ non-negative) for any non-trivial (log) test configurations which satisfies $S_2$ condition. 

$((X,D),L)$ is said to be {\it logarithmically K-polystable} with cone angle $2\pi\beta$ 
when it is logarithmically K-semistable and moreover ${\it DF}_{\beta}((\mathcal{X},\mathcal{D}),\mathcal{L})=0$ for log test configuration which satisifies 
$S_2$ condition if and only if geometrically (without action concerned) 
$(\mathcal{X},\mathcal{D})\cong(X,D)\times\mathbb{A}^1$. Let us call such 
log test configurations, {\it product log test configurations}. 
\end{Def}

Now we recall the formalism of \cite{Od11a}. It is a natural extension of the pioneering \cite{RT07}, which treats $N=1$ case in the notation below. 
Let $(X,L)$ be the polarized varieties with the divisor $D \subset X$ in concern. 
First, let us recall from \cite{Od11a}: 

\begin{Def}[{\cite{Od11a}}]

Let $(X,L)$ be an $n$-dimensional polarized variety. 
A coherent ideal $\mathcal{J}$ of $X\times \mathbb{A}^{1}$ is called a \textit{flag ideal} if 
$\mathcal{J}=I_{0}+I_{1}t+\dots+I_{N-1}t^{N-1}+(t^{N})$,  
where $I_{0}\subseteq I_{1}\subseteq \dots I_{N-1} \subseteq \mathcal{O}_{X}$ is a sequence of coherent ideals. 
(It is equivalent to that the ideal is $\mathbb{G}_{m}$-invariant under the natural action 
of $\mathbb{G}_{m}$ on $X\times \mathbb{A}^{1}$.) 
\end{Def}

Let us recall some notation from \cite{Od11a}. 
We set $\mathcal{L}:=p_{1}^{*}L$ on $X\times \mathbb{A}^{1}$ and $\bar{\mathcal{L}}$ on $X\times \mathbb{P}^1$, 
and denote the $i$-th projection morphism from $X \times \mathbb{A}^{1}$ or $X \times \mathbb{P}^{1}$ by $p_{i}$. 
Let us write the blowing up as $\Pi \colon \bar{\mathcal{B}}(:=Bl_{\mathcal{J}}(X\times \mathbb{P}^{1}))\rightarrow X\times 
\mathbb{P}^{1}$ or its restriction $\Pi \colon \mathcal{B}(:=Bl_{\mathcal{J}}(X\times \mathbb{A}^1))\rightarrow X\times \mathbb{A}^1$, and the natural exceptional Cartier divisor as $E$, i.e.\ $\mathcal{O}(-E)=\Pi^{-1}\mathcal{J}$. 
Denote $Bl_{\mathcal{J}}|_{(D\times \mathbb{A}^1)}(D\times \mathbb{A}^1)$ 
(resp.\  $Bl_{\mathcal{J}}|_{(D\times \mathbb{P}^1)}(D\times \mathbb{P}^1)$) 
as $\mathcal{B}_{(D\times \mathbb{A}^1)}$ (resp.\ $\bar{\mathcal{B}}_{(D\times 
\mathbb{P}^1)}$). 
We also write $\Pi^*\mathcal{L}$ on $\mathcal{B}$ (resp.\ $\Pi^*\bar{\mathcal{L}}$ on $\bar{\mathcal{B}}$) simply as $\mathcal{L}$ (resp.\ $\bar{\mathcal{L}}$). 
Let us assume $\mathcal{L}^{\otimes r}(-E)$ on $\mathcal{B}$ is (relatively) semi-ample (over $\mathbb{A}^{1}$) for $r\in \mathbb{Z}_{>0}$ and 
 consider the Donaldson-Futaki invariant of the blowing up (semi) test configuration 
$(\mathcal{B}, \mathcal{L}^{\otimes r}(-E))$. Note that it has natural $\mathbb{G}_m$ action as we are dealing with flag ideals i.e. $\mathbb{G}_m$-invariant ideal on $X\times \mathbb{A}^1$. Actually these ``a priori special" semi test configurations are sufficient for the  study of log K-stability. 

\begin{Prop}[{cf.\ \cite[Proposition 3.8, 3.10]{Od11a}}]
For a given test configuration $(\mathcal{X},\mathcal{D},\mathcal{L})$ of $(X,L)$ with exponent $r$, we can associate an flag ideal $\mathcal{J}$ such that 
$(\mathcal{B}:=Bl_{\mathcal{J}}(X\times \mathbb{A}^1), 
\mathcal{B}_{(D\times \mathbb{A}^1)}:=Bl_{\mathcal{J}|_{D\times \mathbb{A}^1}}(D\times \mathbb{A}^1), \mathcal{L}^{\otimes r}(-E))$ with semiample $\mathcal{L}^{\otimes r}(-E)$ with the same (log) Donaldson-Futaki invariants. 
\end{Prop}

\begin{proof}
For the (usual) absolute version, i.e., the case without boundary, 
this is proved in \cite{Od11a}. 
As the log Donaldson-Futaki invariant ${\rm DF}_{\beta}$ is usual Donaldson-Futaki 
invariant plus a boundary part, it suffices to prove that boundary part of $\mathcal{X}$ and $\mathcal{B}$ also coincides. But this is straightforward from the construction 
of $\mathcal{B}$ starting with given test configuration $\mathcal{X}$. About the construction, consult \cite{Od11a}, or \cite[section 2]{Mum77} for details. 
\end{proof}

\begin{Cor}[{cf.\ \cite[Corollary 3.11]{Od11a}}]\label{bl.enough}
For $(X,L)$ to be logarithmically K-stable (resp.\ logarithmically K-semistable) with angle $2\pi\beta$, 
it is necessary and sufficient to check  the blow up semi test configurations of the above type, 
i.e., $(\mathcal{B}:=Bl_{\mathcal{J}}(X\times \mathbb{A}^1), 
\mathcal{B}_{(D\times \mathbb{A}^1)}:=Bl_{\mathcal{J}|_{D\times \mathbb{A}^1}}(D\times \mathbb{A}^1), \mathcal{L}^{\otimes r}(-E))$ with $\mathcal{B}$ Gorenstein in codimension $1$ with $r \in \mathbb{Z}_{>0}$ such that $\mathcal{L}^{\otimes r}(-E)$ 
semiample. 
\end{Cor}

\begin{proof}
As in \cite{Od11a}, recall that we can ``partially normalize" $\mathcal{B}$ 
as follows: take the normalization $\nu\colon \mathcal{B}^{\nu}\rightarrow \mathcal{B}$ and take $p\nu\colon 
(\mathcal{C}:=){\it Spec}_{\mathcal{O}_{\mathcal{B}}}(i_{*}\mathcal{O}_{X\times (\mathbb{A}\setminus\{0\})}\cap \mathcal{O}_{\mathcal{B}^{\nu}})
\rightarrow 
\mathcal{B}$, where $i\colon X\times(\mathbb{A}^{1}\setminus\{0\})\hookrightarrow 
X\times \mathbb{A}^{1}$ is the open immersion. Denote the pullback of exceptional Cartier divisor $E$ on $\mathcal{B}$ to $\mathcal{B}^{p\nu}$ by $E':=p\nu^* E$. 
Note that this $\mathcal{B}^{p\nu}$ is also a blow up of flag ideal $\mathcal{J}':=(\Pi \circ p\nu)_*\mathcal{O}(-lE')$ for $l\gg 0$ as well and is Gorenstein in codimension $1$ by \cite[Lemma 3.9]{Od11a}. 

That does not change boundary part and so 
this Corollary follows from \cite[Proposition 5.1, Remark 5.2]{RT07} which works for any (not necessarily normal) 
polarized varieties, although they assumed normality there. 
\end{proof}

Moreover, we have an explicit formula as follows. From now on,  we always assume that $X$ is an equi-dimensional reduced algebraic projective scheme, which is 
$\mathbb{Q}$-Gorenstein (i.e., $K_X$ is $\mathbb{Q}$-Cartier), 
Gorenstein in codimension $1$ (i.e., there is an open dense Gorenstein subset $U$ with 
${\rm codim}(U\subset X)\geq 2$), 
satisfying Serre's $S_2$ condition (which is weaker than normality) 
and $D$ is an effective integral  $\mathbb{Q}$-Cartier Weil divisor, 
unless otherwise stated. 

\begin{Thm}\label{DF.formula}

Let $(X,L)$, $D$, $\beta$ and $\mathcal{B}$, $\mathcal{J}$ be as above. And we assume that exponent $r=1$. 
$($It is just to make the formula easier. For general $r$, put $L^{\otimes r}$ and $\mathcal{L}^{\otimes r}$ 
to the place of $L$ and $\mathcal{L}$. $)$ 
Furthermore, we assume that $\mathcal{B}$ is Gorenstein in codimension $1$. 
Then the corresponding \textit{log Donaldson-Futaki invariants} (multiplied by a positive constant) can be described as follows: 
$$
(n!)((n+1)!){\rm DF}_\beta(\mathcal{B}, \mathcal{L}-E)=-n(L^{n-1}.(K_X +(1-\beta)D))(\bar{\mathcal{L}}-E)^{n+1}
$$
$$
+(n+1)(L^n)((\bar{\mathcal{L}}-E)^n.\Pi^* ((K_X +(1-\beta)D)\times \mathbb{P}^1)$$
$$
+(n+1)(L^n)(((\bar{\mathcal{L}}-E)^n.(K_{\mathcal{B}/((X,(1-\beta)D)\times \mathbb{A}^1)})_{\rm exc})), 
$$ 

\noindent
using intersection numbers on $\bar{\mathcal{B}}$ and $X$, 
where $K_{\mathcal{B}/((X,(1-\beta)D)\times \mathbb{A}^1)})_{\rm exc}$ denotes 
the exceptional parts of 
$K_{\mathcal{B}/((X,(1-\beta)D)\times \mathbb{A}^1)}:=K_{\mathcal{B}}-\Pi^{*}((K_X+(1-\beta)D)\times \mathbb{A}^1)$. 
\end{Thm}

Before the proof, recall our original Donaldson-Futaki invariants' formula: 

\begin{Thm}[{\cite[Theorem 3.2]{Od11a}}]\label{DF.formula.old}
Let $(X,L)$ and $\mathcal{B}$, $\mathcal{J}$ be as above. 
Then the corresponding Donaldson-Futaki invariant $\DF((\mathcal{B}=Bl_{\mathcal{J}}(X\times \mathbb{A}^{1}), \mathcal{L}(-E)))$  can be described as follows: 

\begin{equation*}
(n!)((n+1)!){\rm DF}((\mathcal{B}, \mathcal{L}(-E)))
=-n(L^{n-1}.K_{X})(\bar{\mathcal{L}}(-E))^{n+1}
\end{equation*}
\begin{equation*}
+(n+1)(L^{n})
((\bar{\mathcal{L}}(-E))^{n}.\Pi^{*}(p_{1}^{*}K_{X}))
\end{equation*}
\begin{equation*}
+(n+1)(L^{n})((\bar{\mathcal{L}}(-E))^{n}.K_{\mathcal{B}/X\times \mathbb{A}^{1}}), 
\end{equation*}
\noindent
using intersection numbers on $\bar{\mathcal{B}}$ and $X$, 
where $K_{\mathcal{B}/(X\times \mathbb{A}^1)}:=K_{\mathcal{B}}-\Pi^{*}(K_X\times \mathbb{A}^1)$. 
\end{Thm}

\begin{proof}[proof of Theorem \ref{DF.formula}]
This follows from simple calculation of the boundary part 
$(1-\beta)(a_0\tilde{b}_0-b_0\tilde{a}_0)$ combined with Theorem \ref{DF.formula.old}.
More precisely, we can calculate as follows. 
$b_0=(\bar{\mathcal{L}}-E)^{n+1}$, $\tilde{b}_0=
(\bar{\mathcal{L}}-E)|_{\bar{\mathcal{B}}_{(D\times \mathbb{P}^1)}}^n$ follows from 
the following fact in \cite{Od11a}: 

\begin{Fact}[{\cite[formula after Lemma 3.4]{Od11a}}]\label{w}
$w(m)=\chi (\bar{\mathcal{B}},\mathcal{L}^{\otimes{m}}(-mE))-\chi (X\times \mathbb{P}^{1},\mathcal{L}^{\otimes{m}})+O(m^{n-1})$. 
\end{Fact}
(For estimation of $\tilde{w}(m)$ and calculation of $\tilde{b}_0$, simply apply the formula \ref{w} to $D$ and $\mathcal{J}|_{(D\times \mathbb{A}^1)}$ instead of $X$ and $\mathcal{J}$. )
$a_0=\frac{1}{n!}(L^n)$, $\tilde{a}_0=\frac{1}{(n-1)!}(L|_D^{(n-1)})$ follows from the weak Riemann Roch theorem (cf.\ e.g. \cite[Lemma 3.5]{Od11a}). Using these description of $\tilde{b}_0$, $b_0$, $\tilde{a}_0$, $a_0$, we can derive our formula \ref{DF.formula}. 
We also use $\Pi^*(D\times  \mathbb{P}^1)=\bar{\mathcal{B}}_{(D\times \mathbb{P}^1)}+(\Pi^*(D\times  \mathbb{P}^1))_{\rm exc}$ as $\mathbb{Q}$-divisors on $\bar{\mathcal{B}}$ on the way of the calculation. 
\end{proof}

\noindent
This formula is a natural extension of the intersection formula of Donaldson-Futaki  invariants given in \cite{Wan08}, \cite{Od11a}. Note that we can extend our 
formula \ref{DF.formula} to any log semi-test configurations. It is because 
the procedure of contracting log semi-test configuration 
via its semi-ample line bundle on the total space or 
taking blow up of flag ideal $\mathcal{J}$ of $X\times \mathbb{A}^1$ 
associated to the log test configuration 
do not change the (log) Donaldson-Futaki invariants nor the right hand sides of the formula. (The first  author talked about this, for the absolute case at CIRM, Luminy, in February of 2011. )

This framework using blow up has advantages such as, we can consider the concepts of  \textit{destabilising subschemes} and moreover, the existence of exceptional divisors  helps the estimation in some situation as in section \ref{sec.alpha}. In particular, the  decomposition of the invariant into two parts are important and useful: 
``\textit{canonical divisor part}" which means the sum of first two terms, and  ``\textit{discrepancy term}" which 
is the last term reflecting the singularity (of the pair $(X,(1-\beta)D)$)). 

It is recently explained in \cite{LX11} that there are certain ``pathological test configurations" $\mathcal{X}$, which are characterized by the following conditions.  Their normalizations are trivial though themselves are {\it not} trivial, and Donaldson-Futaki invariants are vanishing for those. Note that such $\mathcal{X}$ should not satisfy Serre's $S_2$ condition nor normality. Thus, if we consider $S_2$ (or normal) test configurations as in \cite{LX11}, then we do not have problems. 

Our arguments work as those pathological test configurations are also characterized by the condition that associated flag ideals are of the form  $\mathcal{J}=t^N\mathcal{O}_{X\times \mathbb{A}^1}$ with $m\in \mathbb{Z}_{\geq0}$,  i.e., the case when blow up morphism $\Pi$ is just the trivial isomorphism. 
In other words, the first author's paper \cite{Od11a} was not accurately written in the sense he ignored the case $\mathcal{J}=(t^N)$ there. 
However, it works for modified K-stability which only concerns $S_2$ test configurations $\mathcal{X}$, whose corresponding flag ideal $\mathcal{J}$ should not be of that trivial form. Y. O apologizes for this inaccuracy. 

\section{Log K-stability of log Calabi-Yau varieties and log canonical models} 

In this section we extend \cite[Theorem 2.6, 2.10]{Od11b} as follows: 

\begin{Thm}\label{CY.gen}

{\rm (i)} Assume $(X, (1-\beta)D)$ is a log Calabi-Yau pair, i.e., $K_X+(1-\beta)D$ is numerically equivalent to zero divisor and it is semi-log-canonical pair (resp.\ kawamata-log-terminal pair). Then, $((X,D),L)$ is logarithmically K-semistable (resp.\ logarithmically K-stable) with cone angle $2\pi\beta$ for any polarization $L$. 

{\rm (ii)} Assume $(X,(1-\beta)D)$ is a semi-log-canonical model, i.e., $K_X+(1-\beta)D$ is ample and  it is a semi-log-canonical pair. Then, $((X,D),K_X+(1-\beta)D)$ and $\beta \in \mathbb{Q}_{>0}$ is log K-stable with cone angle $2\pi\beta$. 

\end{Thm}

\begin{Rem}
Theorem \ref{CY.gen} (i) extends and algebraically recovers \cite[Theorem 1.1]{Sun11}, 
which gave a more differential geometric proof using existence of Calabi-Yau metrics on $D$ for smooth $D$ case.   Also it provides an algebraic counterpart of 
\cite[Theorem1.1]{Bre11} and \cite[Theorem 2]{JMR11}, where K\"ahler-Einstein metrics with cone angle $2\pi\beta$ are constructed on (smooth) log Calabi-Yau and (smooth)  log canonical models.
\end{Rem}

\begin{proof}

It is easy to see that the canonical divisor part vanishes for the case {\rm (i)} as 
our log canonical divisor is zero. For the case {\rm (ii)}, as in \cite{Od11b}, 
the canonical divisor part equals to $((\mathcal{L}^{\otimes r}(-E)).(\mathcal{L}^{\otimes r}(nE)))$ up to positive constant, and it is proven to be positive in \cite[Lemma 2.7, 2.8]{Od11b}.  

Thus, it is enough to prove that the discrepancy term is 
positive (resp.\ non-negative) if $(X,(1-\beta)D)$ is kawamata-log-terminal (resp.\ 
semi-log-canonical). From the inversion of adjunction (Theorem \ref{IA}), 
it follows that all the coefficients of $(K_{\mathcal{B}/((X,(1-\beta)D)\times \mathbb{A}^1)})_{\rm exc}$ are positive (resp.\ non-negative) 
as in \cite[proof of Theorems 2.6, 2.10]{Od11b}. 
From that, using semiampleness of $\mathcal{L}^{\otimes r}(-E)$ and the following fact 
which is proven in \cite{Od11b}, it follows that the discrepancy term is positive (resp.\ non-negative). 
\begin{Fact}[{\cite[inequality (3) in the Proof of Theorem 2.10]{Od11b}}]
The following inequality holds in our setting: $((\bar{\mathcal{L}}^{\otimes r}(-E))^n.E)>0. $
\end{Fact}
\end{proof}

\begin{Rem}
For the case {\rm (i)}, $(\mathcal{B},\mathcal{L}^{\otimes r}(-E))$ have vanishing log Donaldson-Futaki invariant for $\mathcal{J}$ and $r\in \mathbb{Z}_{>0}$ 
if and only if all the exceptional prime divisors supported on 
$\phi^{-1}_*(\mathcal{X}|_{\{0\}})$ in $\mathcal{B}$ have coefficients zero in 
$(K_{\mathcal{B}/((X,(1-\beta)D)\times \mathbb{A}^1)})_{{\rm exc}}$, 
where $\phi \colon \mathcal{B}\rightarrow \mathcal{X}:={\it Proj}\oplus 
(H^0(X\times \mathbb{A}^1,\mathcal{J}^m(p_1^{*}L^{\otimes rm})))$ is the 
natural morphism defined by the semi-ample line bundle $\mathcal{L}^{\otimes r}(-E)$. 
This follows straightforward from our proof. 

In particular, any image of 
such $\Pi$-exceptional prime divisor $E_i$ with ${\rm codim}(\phi_*E_i \subset \mathcal{X})=1$, 
$\Pi(E_i)$ is log canonical center of $(X\times \mathbb{A}^1,(1-\beta)(D\times \mathbb{A}^1)+X\times \{0\})$ which have only finite candidates 
(cf.\ \cite[Proposition 4.7, 4.8]{Amb03}, \cite[Theorem 2.4]{Fuj11}). For example, if $X$ and $D$ are both smooth, all those $E_i$ have $\Pi(E_i)=D\times \{0\}$. 
\end{Rem}

\begin{Rem}
Concerning the finiteness of automorphism groups of polarized log pair ${\rm Aut}((X,D),L):=\{ \sigma \in {\rm Aut}(X) \mid \sigma^* D=D, \sigma^*L \cong L \}$, 
as we argued in \cite{Od11b}, \cite{OS10}, they follows as a special case of 
\cite[Proposition 4.6]{Amb05}. On the other hand, once we know the reductivity as 
analogue of Matsushima's theorem, we can prove the finiteness after Theorem \ref{CY.gen}. 
However, we allow (semi-)log-canonical singularities to Calabi-Yau pair, from which we can only deduce log K-{\it semi}stability, we do not have finiteness of the autmorphism group in general, e.g., $\mathbb{P}^1$ with two reduced points attached. 
\end{Rem}

\section{Log K-stability and alpha invariants}\label{sec.alpha}

In this section, we extend the result of \cite[Theorem 1.4]{OS10} to results for $\mathbb{Q}$-Fano varieties with anti-canonical boundaries. 
On the way, we also recover an algebraic counterpart of \cite[Theorem 1.8]{Ber10}. 

First recall the definition of {\it global log canonical threshold} (defined in algebro-geometric terms) and the {\it alpha invariant} (defined in analytic terms), which is proven to be equivalent. The definition of the {\it global log canonical threshold} is the following, 
which we use. 

\begin{Def}\label{eq:alpha_finite} 
Assume $(X,D)$ is a log canonical pair with $D$ $\mathbb{Q}$-Cartier, 
which we allow to be $\mathbb{Q}$-divisor in this definition. Set: 

	$${\rm glct}((X,D);L)
	:=
	{\rm inf}{_{m\in\mathbb{Z}_{>0}}}
{\rm inf}{_{
E\in |mL|}
}
{\rm lct}\big((X,D), \frac{1}{m}E \big), $$
which we call the {\it global log canonical threshold} of the pair $(X,D)$ with respect to the polarization $L$. If $D=0$, we simply write ${\rm glct}(X;L)$. 
Here, by the definition of the usual log canonical threshold, 

$${\rm lct}((X,D),\frac{1}{m}E)):={\rm sup}\{ \alpha\mid (X,D+\frac{\alpha}{m}E) \mbox{ is log canonical} \}. $$
\end{Def}

The definition of the {\it alpha invariant} is the following. 
It is first defined by \cite{Tia87} and its natural extension to log setting is 
also discussed in \cite[section 6]{Ber10}. 

\begin{Def}
Assume $X$ is smooth and $(X,D)$ is a klt pair for an effective $\mathbb{Q}$-divisor $D$ in this definition. Write $D$ as $D=\sum d_iD_i$ where $D_i$ are prime divisors  locally defined by $(f_i=0)$. 
Let $\omega$ be a fixed K\"ahler form with K\"ahler class $c_{1}(L)$. 
Let $P(X,\omega)$ be the set of K\"ahler potentials defined by
$$
	P(X,\omega):=\{\varphi\in C^2_\mathbb{R}(X) \mid \sup\varphi=0,\,\omega+\frac{\sqrt{-1}}{2\pi}\partial\bar\partial \varphi >0\}, 
$$
\noindent
where $C^2_\mathbb{R}(X)$ means a space of real valued continuous function of $X$ 
of class $C^2$. The definition of {\it alpha invariant} of $(X,D)$ with respect to the polarization $L$ is: 
$$
	\alpha((X,D);L):=\sup\{\alpha \mid 
e^{-\alpha\varphi}\prod |f_i|^{-2\alpha d_i}
\mbox{ is locally integrable for all } \varphi\in P(X,\omega)\}. 
$$
\end{Def}

\noindent
This is independent of the choice of $\omega$. 
It is known that these notions are equivalent as follows: 
\begin{Fact}[{\cite[Appendix A]{CSD08}, 
\cite[section 6]{Ber10}}]

${\rm glct}((X,D);L)=\alpha((X,D);L)$ for 
klt pair $(X,D)$ with smooth $X$ and polarization $L$. 
\end{Fact}

\noindent
It depends on the approximation theory of pluri-subharmonic functions. 
Consult \cite[Appendix A]{CSD08}, 
\cite[section 6]{Ber10} for the details. 

\begin{Thm}\label{alpha}
For a $\mathbb{Q}$-Fano variety $X$ (i.e. $-K_X$ ample) and anti-canonical effective integral reduced  $\mathbb{Q}$-Cartier divisor $D$, which form a purely log terminal pair (resp.\ semi-log-canonical pair) $(X,D)$, 
if ${\rm glct}((X,(1-\beta)D);-K_X)>(resp.\ \geq )(n/n+1)\beta$ then it is logarithmically  K-stable (resp.\ logarithmically K-semistable) with cone angle $2\pi\beta$. 
\end{Thm}

\noindent
We note that $\beta=0$ case also follows from Theorem \ref{CY.gen} {\rm (i)}. 
Using the original alpha invariant, we state a weaker result as follows. 
This corresponds to the analytic statement of \cite[Theorem 1.8]{Ber10}. 

\begin{Cor}\label{weak.alpha}
For the above setting, we further assume that $D$ is irreducible and Cartier. 
Then, $(X,-K_X)$ is logarithmically K-stable (resp.\ logarithmically K-semistable) for cone  angle $2\pi\beta$ with $0<\beta<(\frac{n+1}{n}){\rm min}\{{\rm glct}(X;-K_X), {\rm glct}(D;-K_X|_D)\}$ (resp.\ $0\leq\beta\leq(\frac{n+1}{n}){\rm min}\{{\rm glct}(X;-K_X), {\rm glct}(D;-K_X|_D)\}$). 
\end{Cor}

\noindent
Please be careful that $-K_X$ (resp.\ $-K_X|_D$) appeared in the global log canonical 
threshold above mean {\it polarization}, but not boundary divisors attached to ambient variety $X$ (resp.\ $D$). For the proof, 
we follow the viewpoint of ``S-stability" introduced in \cite{Od11c}. 
\begin{proof}[Proof of Theorem \ref{alpha}] 

It follows from the formula \ref{DF.formula} that our log Donaldson-Futaki invariant 
${\rm DF}_{\beta}(\mathcal{B},\mathcal{L}^{\otimes r}(-E))$ is 
\begin{equation}\label{Fano.DF}
-\beta(L^n)((\bar{\mathcal{L}}-E)^n.\bar{\mathcal{L}})+\\ 
(L^n)((\bar{\mathcal{L}}(-E))^n.(n+1)r(K_{\mathcal{B}/((X,(1-\beta)D)\times \mathbb{A}^1)})_{\rm exc}-nE), 
\end{equation}
where $r$ is the exponent of the log semi-test configuration $(\mathcal{B},\mathcal{L}^{\otimes r}(-E))$. 
As \cite[Proposition 4.3]{OS10} proved the first term is always non-negative 
it is enough to show that all the coefficients of exceptional prime disivor 
$(n+1)r(K_{\mathcal{B}/((X,(1-\beta)D)\times \mathbb{A}^1)})_{\rm exc})-nE$ 
is positive (resp.\ non-negative) under the assumption of global log canonical threshold that 
\begin{equation}\label{glct.cond}
{\rm glct}((X,(1-\beta)D),-K_X)>(resp.\ \geq )(n/n+1)\beta. 
\end{equation}

To prove it, we need the following inequalities of discrepancies for any exceptional prime divisor $E_i$ on 
$\mathcal{B}$ in concern: 
$$
a(E_i;(X\times \mathbb{A}^1,(1-\beta)(D\times \mathbb{A}^1)+\frac{n\beta}{r(n+1)}\mathcal{J}+X\times \{0\})) 
$$
$$\geq
a(E_i;(X\times \mathbb{A}^1,(1-\beta)(D\times \mathbb{A}^1)+\frac{n\beta}{r(n+1)}I_0+X\times \{0\})) 
$$
$$
\geq 
a(E_i;(X\times \mathbb{A}^1,(1-\beta)(D\times \mathbb{A}^1)+\frac{n\beta}{r(n+1)}(F\times \mathbb{A}^1)+X\times \{0\}). 
$$

\noindent
Here, $F$ is taken to be an effective $\mathbb{Q}$-divisor which corresponds to an arbitrary non-zero holomorphic section of $H^0(X,I_0^{m}(-rmK_X))$. 
That vector space does not vanish for sufficiently divisible positive integer $m$ 
as our assumption of semi-ampleness of $\mathcal{L}^{\otimes r}(-E)$ 
says $H^0(\mathcal{B},\mathcal{L}^{\otimes rm}(-mE))=H^0(X\times \mathbb{A}^1,
\mathcal{J}^m\mathcal{L}^{\otimes rm})$ generate $\mathcal{J}^m\mathcal{L}^{\otimes rm}$  and $H^0(X,I_0^{m}(-rmK_X))$ 
is just the subspace of $H^0(X\times \mathbb{A}^1,
\mathcal{J}^m\mathcal{L}^{\otimes rm})$ which is fixed by the $\mathbb{G}_m$-action. 
Note that the discrepancy on the first term and second term involve ideal (not necessarily corresponding to divisor) but recall that we can define discrepancy completely similarly in this case 
as we noted in subsection \ref{disc}. 
The comparison between the first term and the second term simply follows 
from $I_0\subset \mathcal{J}$ and the last 
inequality follows from the definition of $E$. 
Note that it is enough to show that the first term 
is bigger than $-1$ (resp.\ at least $-1$) so 
we only need to prove $(X,(1-\beta)D+\frac{n\beta}{n+1}E)$ is purely log terminal 
(resp.\ log canonical) by the inversion of adjunction of 
log-terminality and log-canonicity (Theorem \ref{IA}). 

On the other hand, the condition (\ref{glct.cond}) implies those. 
This completes the proof of Theorem \ref{alpha}. 
\end{proof}

\begin{proof}[Proof of Corollary \ref{weak.alpha}] 
Decompose $E$ which appeared in the proof of Theorem \ref{alpha} 
as $F=aD+F'$ with some $0\leq a \leq 1$ such that 
${\rm Supp}(F')$ does not include $D$. 
Then, to see the kawamata-log-terminality (resp.\ log-canonicity) of 
$(X,(1-\beta)D+\frac{n}{n+1}\beta F)$, it is sufficient to prove 
pure-log-terminality (resp.\ log-canonicity) of 
$(X,D+\frac{n}{n+1}\beta F')$. 
Note that for log terminal version, we assumed $\beta>0$. 

On the other hand, our assumptions imply the following two.  

\begin{Claim}\label{cl}

{\rm (i)}$(X, (\frac{n}{n+1})\beta F')$ is klt (resp.\ lc). 

{\rm (ii)}$(D, (\frac{n}{n+1})\beta F'|_D)$ is also klt (resp.\ lc). 

\end{Claim}

\noindent
Indeed, the condition {\rm (i)} follows from the condition $\beta<(\frac{n+1}{n}){\rm glct}(X;-K_X)$ and the condition {\rm (ii)} follows from the condition $\beta<(\frac{n+1}{n}){\rm glct}(D; -K_X|_D)$. 

Claim \ref{cl} {\rm (i)} implies that $(X\setminus D,(D+\frac{n}{n+1}\beta E')|_{(X\setminus D)}=(\frac{n}{n+1}\beta E')|_{(X\setminus D)})$ is klt (resp.\ lc) and 
the second condition {\rm (ii)} implies 
$(X,(1-\beta)D+\frac{n}{n+1}\beta E')$ is plt (resp.\ lc) on an open neighborhood of $D$, 
due to the inversion of adjunction \ref{IA}. Combining together, we obtain that 
$(X,(1-\beta)D+\frac{n}{n+1}\beta E')$ is plt (resp.\ lc) as we wanted. 

\end{proof}

\begin{Rem}
If we allow $D$ to be not necessarily Cartier, we obtain similar results 
by considering pair $(D,{\rm Diff}_D(0))$ and associated global log canonical thresholds, 
instead of those of single $D$. Here, ${\rm Diff}_D(0)$ is a {\it different}, which is a divisor of $D$ encoding the failure of adjunction (cf.\ e.g. \cite{Kaw07}). Also we can extend 
to the case where $D$ is not necessarily normal nor Cartier. For that case, we need to 
think global log canonical threshold ${\rm glct}(D,-K_X|_D)$ on the normalization of $D^{\nu}$ with different of conductor divisor ${\rm cond}(\nu)$ attached i.e. ${\rm glct}((D^{\nu},{\rm Diff}_{D^{\nu}}({\rm cond}(\nu)));\nu^*(-K_X|_D))$ instead. 
\end{Rem}

\begin{Rem}
Assume $(X,-K_X)$ is K-stable in the absolute sense, then if we allow $\beta>1$ and consider logarithmic K-stability (resp.\ logarithmic K-semistability) in the same way as in Definitions \ref{DF}, \ref{Kst}, 
$\beta<(resp.\ \leq)(\frac{n+1}{n}){\rm glct}(X)$ simply implies 
log K-stability (resp.\ log K-semistability) with cone angle $2\pi\beta$. 
This is because sub kawamata-log-terminality (resp.\ sub log-canonicity) 
condition of $(X,(1-\beta)D+(\frac{n}{n+1})\beta E)$ implies log K-stability 
(resp.\ log K-semistability) as in the proof of Corollary \ref{weak.alpha} and 
$(1-\beta)D<0$ so that we can simply ignore that term. It is interesting that this 
bound does not depend on $D$. 
\end{Rem}

An easy consequences of Theorem \ref{alpha} is 
\begin{Cor} 
There is no algebraic subgroup of ${\rm Aut}(X, D)$ isomorphic to $\mathbb{G}_m$.  
\end{Cor}

\noindent
Here, ${\rm Aut}(X,D):=\{ \sigma \in {\rm Aut}(X) \mid \sigma^* (D)=D \} \subset 
{\rm Aut(X)}$ is the automorphism group of the pair 
(cf.\ \cite[Proposition 4.6]{Amb05}). 

\begin{proof}
If there is such a subgroup and consider one corresponding  
non-trivial one paramter subgroup $\lambda \colon \mathbb{G}_m 
\rightarrow {\rm Aut}(X,D)$, then at least one of log Donaldson-Futaki invariants of 
product log test configurations coming from $\lambda$ or $\lambda^{-1}$ should be 
negative as the sum of two is zero. 
\end{proof}

Now given Theorem  \ref{alpha}, one can define for any pair $(X, D)$ an invariant
$$
\beta(X, D):=\sup\{\beta>0|(X, D)\ \text{is log K-stable with cone angle}\ 2\pi\beta\}.
$$
\noindent
It is well defined as, if we take $\mathcal{J}$ which corresponds to maximal ideal 
of $(p,0)\in D\times \{0\}$, then we have 
${\rm DF}_\beta<0$ for $\beta\gg 0$ 
as first part of formula (\ref{Fano.DF}) $-\beta(L^n)((\bar{\mathcal{L}}^{\otimes r}(-E)^n.\mathcal{L})$ vanishes and the second part of formula (\ref{Fano.DF}) goes to $-\infty$ as $\beta\rightarrow\infty$. 
We have also proved that $$\beta(X, D)\geq \biggl(\frac{n+1}{n}\biggr){\rm min}\{{\rm glct}(X;-K_X), {\rm glct}(D;-K_X|_D)\} . $$ 

\noindent
In particular, we proved $\beta(X,D)$ is a positive number. 

The following corollary is a simple application of definition of log-K-stability. 
\begin{Cor}
The pair $(X, D)$ is logarithmically K-stable for $\beta\in(0, \beta(X, D))$, and logarithmically K-unstable for $\beta>\beta(X, D)$. 
\end{Cor}

We expect the pair $(X,D)$ is logarithmically K-semi-stable for $\beta=\beta(X,D)$. This fits in with the conjecture of Donaldson \cite{Don11}, in terms of existence of K\"ahler-Einstein metrics on $X$ with cone singularities along $D$.  More precisely,  when $X$ and $D$ are smooth, one can also define an invariant
$$
R(X, D):=\sup\{\beta>0|\exists\ \text{a KE metric on}\ X \text{with cone angle}\ 2\pi \beta \ \text{along } D)\}. 
$$
The logarithmic version of the Yau-Tian-Donaldson conjecture would suggest that $\beta(X, D)=R(X, D)$. 

\section{Log K-stability and semi-log-canonicity}

In this section, we generalize \cite[Theorem 1.1, 1.2]{Od08} as follows. 

\begin{Thm}\label{slc}

{\rm (i)}
If a log polarized variety 
$((X,D),L)$ is logarithmically K-semistable with cone angle $2\pi\beta$, then $(X,(1-\beta)D)$ is 
semi-log-canonical pair. 

{\rm (ii)}
If a log $\mathbb{Q}$-Fano anti-(pluri-)canonically polarized variety 
$((X,D),L)$ is logarithmically K-semistable with cone angle $2\pi\beta$ and $L=\mathcal{O}_{X}(-m(K_X+(1-\beta)D))$ with $m\in \mathbb{Z}_{>0}$, 
then $(X,(1-\beta)D)$ is kawamata-log-terminal pair with $\beta>0$. 
\end{Thm}

\begin{proof}

We prove it in completly similar way as in \cite{Od08}. Assume the contrary. 
First, we argue version {\rm (i)}. 
As a first step of the proof, take the semi-log-canonical model $\pi\colon B=Bl_{I}(X)\rightarrow (X,(1-\beta)D)$ of $(X,(1-\beta)D)$, which is possible by \cite{OX11}. 
Then, all the coeffecients of $(K_{B/((X,(1-\beta)D)})_{\rm exc}$ is less than $-1$ 
by the negativity lemma (cf.\ \cite[Lemma 3.38]{KM98}). 
Second, if $X$ is normal we take $\mathcal{J}:=\overline{((I+(t^l))^N)}$ where $\bar{}$ 
denotes integral closure of ideal, $l$ is sufficiently divisible positive integer and $N\gg 0$.  If $X$ is non-normal, as similarly as in 
\cite[section5]{Od08} or Corollary \ref{bl.enough}, we first 
take partial normalization $\mathcal{B}^{p\nu}$ 
of ${\rm Bl}_{I+(t^l)}(X\times \mathbb{A}^1)$ with sufficiently divisible $l\in \mathbb{Z}_{>0}$ and 
take corresponding flag ideal $\mathcal{J}$ whose blow up is $\mathcal{B}^{p\nu}$. 

Then, as in \cite{Od08}, all the coefficients of 
$(K_{\mathcal{B}/((X,(1-\beta)D)\times \mathbb{A}^1)})_{{\rm exc}}$ 
are negative for this $\mathcal{J}$ so that 
${\it DF}_{\beta}(\mathcal{B},\mathcal{L}^{\otimes r}(-E))<0$ for $r \gg 0$ 
by the formula \ref{DF.formula}. 
Hence, this implies logarithimic K-unstability of $((X,D),L)$ with cone angle $2\pi\beta$. 
This completes the proof for the general case {\rm (i)}. 

In the case {\rm (ii)}, if $(X,(1-\beta)D)$ is semi-log-canonical but 
not klt in codimension $1$ 
or not normal, then we can take 
a flag ideal $\mathcal{J}$ with ${\rm Supp}(\mathcal{O}/\mathcal{J})$ has dimension $n-1$ which is included in  non-kawamata-log-terminal locus or non-normal locus, 
such that $\mathcal{B}$ is Gorenstein in codimension $1$ (otherwise, take partial normalization) the coefficients of an exceptional prime divisor $E_i$ of  
$(K_{\mathcal{B}/((X,(1-\beta)D)\times \mathbb{A}^1)})_{{\rm exc}}$ is $0$ if 
${\rm dim}(\Pi(E_i))=n-1$.  
(Recall that we did similar procedure in \cite[section 6]{Od08}). 
In this case, the leading coefficient of 
${\rm DF}_{\beta}(\mathcal{B},\mathcal{L}^{\otimes r}(-E))$ with respect to the variable $r$ is $(\mathcal{L}^{n-1}.E^2)<0$. 

Thus, we can assume that $(X,(1-\beta)D)$ is klt in codimension $1$. 
Assume that it is not klt. Then, we can take non-trivial flag ideal $\mathcal{J}$ 
with 
$(K_{\mathcal{B}/((X,(1-\beta)D)\times \mathbb{A}^1)})_{{\rm exc}}=0$ in the same way 
as for the case {\rm (i)}. 
On the other hand, in this log $\mathbb{Q}$-Fano case {\rm (ii)}, the canonical divisor  part is always negative so that the whole log Donaldson-Futaki invariant is also negative. 
This completes the proof of the log $\mathbb{Q}$-Fano case ${\rm (ii)}$. 
\end{proof}

\begin{Rem}
Note that the case {\rm (ii)} discussed above corresponds to log $\mathbb{Q}$-Fano case, which are more general than the pair we discussed in section \ref{sec.alpha} i.e., 
$\mathbb{Q}$-Fano varieties with anti-canonical boundaries. 
\end{Rem}

\noindent
By combining Theorem \ref{CY.gen} and Theorem \ref{slc} {\rm (i)}, 
we get the following. 

\begin{Cor}
{\rm (i)}
Assume $(X, (1-\beta)D)$ is a log Calabi-Yau pair, i.e., $K_X+(1-\beta)D$ is numerically equivalent to zero divisor with a polarization $L$. Then,  $((X,D),L)$ is logarithmically K-semistable with cone angle $2\pi\beta$ 
if and only if $(X,(1-\beta)D)$ is a semi-log-canonical pair. 

{\rm (ii)}
Assume $(X,(1-\beta)D)$ is (pluri-)log canonically polarized, i.e., $K_X+(1-\beta)D$ is ample and $L=\mathcal{O}_{X}(m(K_X+(1-\beta)D))$ for $m\in \mathbb{Z}_{>0}$. 
Then, the following three conditions are equivalent: 

{\rm (a)} $((X,D),L)$ is log K-stable with cone angle $2\pi\beta$,  

{\rm (b)} $((X,D),L)$ is log K-semistable with cone angle $2\pi\beta$, 

{\rm (c)} $(X,(1-\beta)D)$ is semi-log-canonical. 
\end{Cor}


\begin{thebibliography}{99}

\bibitem[Amb03]{Amb03}
F.~Ambro, Quasi-log varieties, Tr.\ Mat.\ Inst.\ Steklova 
vol.\ \textbf{240} 220-239, (2003); 

English translation in Proc.\ Steklov Inst.\ Math.\ vol.\ \textbf{240}, 214–233 (2003). 

\bibitem[Amb05]{Amb05}
F.~Ambro, 
The moduli $b$-divisor of an lc-trivial fibration. (English summary) 
Compos.\ Math.\ vol.\ \textbf{141}, 385–403 (2005). 

\bibitem[Ber10]{Ber10}

R.~J.~Berman, 
A thermodynamical formalism for Monge-Ampere equations, Moser-
Trudinger inequalities and Kahler-Einstein metrics, arXiv:1011.3976. 

\bibitem[Bre11]{Bre11}
S.~Brendle, 
Ricci-flat K\"ahler metrics with edge singularities, 
arXiv:1103.5454. 

\bibitem[CSD08]{CSD08}

I.~A.~Cheltsov, K.~A.~Shramov, 
Log canonical thresholds of smooth Fano threefolds, 
With an appendix by Jean-Pierre Demailly, 

English translation: 
Russian Mathematical Surveys, vol.\ \textbf{63}, no. 5, 859–958 (2008).  

\bibitem[Don02]{Don02}

S.~Donaldson, 
Scalar curvature and stability of toric varieties, 
J.\ Diff.\ Geom.\ (2002). 

\bibitem[Don11]{Don11}

S.~K.~Donaldson, 
K\"ahler metrics with cone singularities along a divisor, arXiv:1102.1196. 

\bibitem[Fuj11]{Fuj11}
O.~Fujino, 
Non-vanishing theorems for log canonical pairs, 
J.\ Alg.\ Geom.\ vol.\ \textbf{20}, 771–783 (2011). 

\bibitem[JMR11]{JMR11}

T.~Jeffres, R.~Mazzeo, Y.~Rubinstein, 
K\"ahler-Einstein metrics with edge singularities, arxiv: 1105.5216.

\bibitem[Kaw07]{Kaw07}
M.~Kawakita, 
Inversion of adjunction on log canonicity,
Invent.\ Math.\ vol.\ \textbf{167}, 129-133, (2007).  

\bibitem[KM98]{KM98}

J.~Koll\"ar, S.~Mori, 
Birational geometry of algebraic varieties, Cambridge University Press (1998). 

\bibitem[Li11]{Li11}
 C. Li,
 Remarks on logarithmic K-stability,
  arxiv: 1104.0428. 



\bibitem[LX11]{LX11}
C.~Li, C.~Xu, 
Special test configurations and K-stability of Fano varieties, 
arXiv:1111.5398. 

\bibitem[LX]{LX}
C.~Li, C.~Xu, 
in preparation. 

\bibitem[Mum77]{Mum77}
D.~Mumford, 
Stability of Projective Varieties, 
Enseignement Math.\ vol.\ \textbf{23} (1977). 

\bibitem[Od08]{Od08}

Y.~Odaka, 
On the GIT stability of polarized varieties via Discrepancy, 
preprint. 

Newest submitted version is available at \texttt{http://www.kurims.kyoto-u.ac.jp/\~{}yodaka}.  

\bibitem[Od11a]{Od11a}

Y.~Odaka, 
A generalization of Ross-Thomas slope theory, 
to appear in Osaka J. Math. 

\bibitem[Od11b]{Od11b}
Y.~Odaka, 
The Calabi conjecture and K-stability, 
Int.\ Math.\ Res. Notices, no. \textbf{13},  (2011). 

\bibitem[Od11c]{Od11c}

Y.~Odaka, 
On the K-stability of ($\mathbb{Q}$-)Fano varieties, 
slide of the talk at Kyoto Symposium, Complex analysis, July, 2011. 
available at \texttt{http://www.kurims.kyoto-u.ac.jp/\~{}yodaka/Kyoto.Hayama.pdf}

\bibitem[OS10]{OS10}

Y.~Odaka, Y.~Sano, 
Alpha invariants and K-stability of $\mathbb{Q}$-Fano varieties, 
arXiv:1011.6131. 

\bibitem[OX11]{OX11}

Y.~Odaka, C.~Xu, 
Log canonical models of singular pairs and its applications, 
arXiv:1108.1996. 

\bibitem[RT07]{RT07}

J.~Ross, R.~Thomas, 
A study of Hilbert-Mumford criterion of stability of projective varieties, 
J.\ Alg.\ Geom.\ (2007). 

\bibitem[Sun11]{Sun11}

S.~Sun, 
Note on K-stability of pairs, 
arXiv:1108.4603. 

\bibitem[Tia87]{Tia87} G. Tian, On K\"ahler-Einstein metrics on certain K\"{a}hler manifolds with $C\sb 1(M)>0$, Invent.\ Math.\ vol.\ \textbf{89},  (1987) 225--246. 

\bibitem[Tia97]{Tia97}

G.~Tian, 
K\"ahler-Einstein metrics with positive scalar curvature, 
Invent.\ Math.\ vol.\ \textbf{130}, 1-37, (1997). 

\bibitem[Wan08]{Wan08}

X.~Wang, 
Heights and GIT weights, 
preprint. available on his webpage: \texttt{http://www.math.cuhk.edu.hk/\~{}xiaowei/} 

\end{thebibliography}
\end{document}